\documentclass[11pt]{amsart}   %%%%%%%Please do not change
\usepackage{amsfonts,latexsym,amssymb
}

\newtheorem{theorem}{Theorem}[section]

\newtheorem{proposition}[theorem]{Proposition} 
 
\newtheorem{corollary}[theorem]{Corollary} 
\newtheorem{fact}[theorem]{Fact} 

\theoremstyle{definition}
\newtheorem{definition}[theorem]{Definition}

\newtheorem{example}[theorem]{Example}

\theoremstyle{remark}
\newtheorem{remark}[theorem]{Remark}

\DeclareMathOperator{\cf}{cf}
\DeclareMathOperator{\CAP}{CAP}

\begin{document}
%%%%%%%%%%%%%%%%%%%%%%%%%%%%%%%%%%%%%%%%%%%%%%%%%%%%%%%%%%%%
%%%%%%%%%%%%%%%%%%%%%%%%%%%%%%%%%%%%%%%%%%%%%%%%%%%%%%%%%%%%
% This a placeholder for the TOPLOGY PROCEEDINGS logo %%%%%%
\noindent                                             %%%%%%
\begin{picture}(150,36)                               %%%%%%
\put(5,20){\tiny{Submitted to}}                       %%%%%%
\put(5,7){\textbf{Topology Proceedings}}              %%%%%%
\put(0,0){\framebox(140,34){}}                        %%%%%%
\put(2,2){\framebox(136,30){}}                        %%%%%%
\end{picture}                                         %%%%%%
%%%%%%%%%%%%%%%%%%%%%%%%%%%%%%%%%%%%%%%%%%%%%%%%%%%%%%%%%%%%
%%%%%%%%%%%%%%%%%%%%%%%%%%%%%%%%%%%%%%%%%%%%%%%%%%%%%%%%%%%%

\vspace{0.5in}

\title [pseudocompactness, ultrafilter convergence]
{Some compactness properties related to pseudocompactness
and ultrafilter convergence}

\author[]{Paolo Lipparini} 
\address{Partimento di Matematica\\
Viale della Ricerca Scientifica\\
II Universit\`a di Roma (Tor Vergata)\\
I-00133 ROME 
ITALY}
\urladdr{http://www.mat.uniroma2.it/\textasciitilde lipparin}
\thanks{The author has received support from MPI and GNSAGA.
We wish to express our gratitude to X. Caicedo
and S. Garcia-Ferreira for stimulating discussions and correspondence} 

\subjclass[2000]{Primary 54A20, 54D20;
Secondary 54B10}
\keywords{Pseudocompactness, $D$-pseudocompactness, ultrafilter convergence, $D$-limit point, complete accumulation point, $\CAP_ \lambda $, $ [\mu, \lambda ] $-compactness, productivity of compactness, family of subsets of a topological space}

\begin{abstract} 
We discuss some notions of compactness and convergence 
relative to a specified
 family $\mathcal F$ 
of subsets of some topological space $X$. 
The two most interesting particular cases
of our construction appear to be the following ones.

(1) The case in which $\mathcal F$ 
is the family of all singletons of $X$, in which case we get back the more usual notions.

(2) The case in which $\mathcal F$ 
is the family of all nonempty open subsets of $X$,
in which case we get notions related to pseudocompactness.

A large part of the results in this note are known in particular case (1);
 the results are, in general, new in case (2). As an  example, we   characterize
 those spaces which are $D$-pseudocompact, for
some ultrafilter $D$ uniform over $\lambda$.
\end{abstract}

\maketitle  

\section{Introduction} \label{intronuo} 

In this note we study various compactness and convergence
properties relative to a family $\mathcal F$ of subsets of some topological space.
In particular, we relativize 
to $\mathcal F$ the notions
of $D$-compactness, $\CAP_ \lambda $, and 
$ [\mu, \lambda ] $-compactness.
The two particular cases which motivate our treatment are when $\mathcal F$ is either (1) the family of all singletons of $X$, or (2) the family of all nonempty open sets of $X$. As far as case (2) is concerned, we can equivalently consider 
nonempty elements of some base, and we can also equivalently consider
those
sets which are the closure of some nonempty open set.

Our results concern the mutual relationship among the above compactness 
properties, and their behavior with respect to products.
Some results which are known in 
particular case (1) are generalized to the case of an arbitrary family
$\mathcal F$. Apparently, a few results are new even in particular case (1). 

Already in particular case (2),  our results appear to be new.
For example, we get characterizations
of those spaces which are $D$-pseudocompact, for
some ultrafilter $D$ uniform over $\lambda$ (Corollary \ref{pseudofprod}).
 
Similarly, we get equivalent conditions for the weaker 
local form   asserting that,
for every $ \lambda $-indexed family of nonempty sets of $X$,
there exists some uniform ultrafilter $D$ over $\lambda$ such that the family has some $D$-limit
point in $X$ (Theorem \ref{equivcpn}). In the particular
case $\lambda= \omega $, we get nothing but more conditions
equivalent to pseudocompactness (for Tychonoff spaces).

At first reading, the reader might
consider only the above particular cases (1) and (2), and look at this note as
a generalization to pseudocompactness-like notions of results
already known about ultrafilter convergence and complete accumulation
points.
Of course, it might be the case that
our  definitions and results can be applied to other situations, apart from the two mentioned particular ones; however,
 we have not worked details yet.

\smallskip

No separation axiom is assumed, unless
explicitly mentioned.

\subsection{Some history and our main aim} \label{intro} 

The notion of (pointwise) ultrafilter convergence has proven particularly
useful in topology, especially in connection with the study of compactness
properties and existence of complete accumulation points, not excluding many other kinds of applications.
In particular, ultrafilter convergence is an essential tool in studying compactness properties
of products. In a sense made precise in \cite{topproc},  the use of ultrafilters
is unavoidable in this situation.

Ginsburg and Sack's 1975 paper \cite{GS} is a pioneering work in applications
of pointwise ultrafilter convergence. In addition, \cite{GS} introduces a fundamental
new tool, the idea of considering ultrafilter limits of subsets (rather than points)
of a topological space. In particular, taking into consideration
ultrafilter limits of nonempty open sets provides
 deep applications to pseudocompactness, as well as
 the possibility of introducing further pseudocompactness-like notions.
Some analogies, as well as some differences between the two cases were already discussed in \cite{GS}. Subsequently,
\cite{GF1} analyzed in more details some analogies. 

Ginsburg and Sack's work concentrated on ultrafilters uniform over $ \omega $.
Generalizations and improvements for ultrafilters over larger cardinals appeared, for example, 
in  \cite{Sa} in the case of pointwise $D$-convergence,
and in \cite{GF} in the case of $D$-pseudocompactness.

A new wave of results, partially inspired by seemingly unrelated 
 problems in Mathematical Logic, arose when Caicedo \cite{Cprepr, C},
using ultrafilters, proved some two-cardinals transfer results for compactness  of products. For example, among many other things, Caicedo proved that 
if all powers of some topological space $X$ are
$ [ \lambda ^+, \lambda ^+] $-compact, then 
all   powers of $X$ are $ [ \lambda , \lambda ] $-compact.
Subsequently, further results along this line appeared in \cite{topproc,topappl,nuotop}. 

The aim of this note is twofold. First, we provide analogues,
 for pseudocompactness-like notions, of results previously
proved only for pointwise convergence; in particular, we provide versions of 
many results
appeared in \cite{Cprepr, C, topproc,topappl}.
 
Our second aim is to insert the two above-mentioned kinds of results
into a more general framework. 
Apart from the advantage of a unified treatment
of both cases, we hope that this abstract approach will contribute to put in a clearer
light the methods and notions used in the more familiar case of
pointwise convergence. Moreover, as we mentioned, 
\cite{GS} noticed certain analogies between the two cases,
but noticed also that  there are asymmetries.
In our opinion, our treatment provides a very neat explanation
 for such asymmetries. See the   discussion below in subsection
\ref{syn} 
relative to
Section \ref{behprod}.

Finally, let us mention that,   for particular case (1), a large part
 of the results presented here is well known; however,
even in this particular and well studied case,
we provide a couple of results which might be new: see, e. g.,
Propositions \ref{singreg} and \ref{singreg2}, and
 Remark \ref{res}.

\subsection{Synopsis} \label{syn} 

In detail, the paper is divided as follows.

In Section \ref{dcomp} we introduce the notion of
$D$-compactness relative to some family $\mathcal F$ of subsets
of some topological space $X$. This provides a common generalization
both of pointwise $D$-compactness, and of $D$-pseudocompactness as introduced by \cite{GS,GF}.
Some trivial facts hold about this notion: for example, we can equivalently consider
the family of all the closures of elements of $\mathcal F$.

In Section \ref{capfs} we discuss the notion
of  a complete accumulation point relative to $\mathcal F$. 
In fact, two version are presented: the first one, starred, dealing with 
\emph{sequences} of subsets, and the second one, unstarred, dealing with 
\emph{sets} of subsets. That is, in the starred case repetitions are allowed, while 
they are not allowed
in the unstarred case.
The difference between the two cases is essential only when dealing
with singular cardinals (Proposition \ref{singreg}).
In the classical case when $\mathcal F$ is the set of all singletons, the unstarred notion is  most used 
in the literature; however, we show that
the exact connection between 
the notion of a $D$-limit point and the existence of a complete accumulation point   holds
only for the starred variant (Proposition \ref{ufacc}).

In
Section \ref{rel}
we introduce a generalization of 
$[\mu, \lambda ]$-compactness which also depends on $\mathcal F$, and
in Theorem \ref{equivcpn}  we prove the equivalence 
among  many of the $\mathcal F$-dependent notions we have defined
before.

Section \ref{behprod} discusses 
the behavior of the above notions in connection with (Tychonoff)  products.
Actually, for sake of simplicity only, we mostly deal with powers.
Since, in our notions, a topological space $X$ comes equipped with a
family $\mathcal F$ of subsets  
attached to it, we have to specify which family should be
attached to the power $X^ \delta $. In order to get
significant results, the right choice is to attach to
$X^ \delta $ the family $\mathcal F ^ \delta $ consisting of all products of
$ \delta $ members of $\mathcal F$ (some variations are possible). In the case when
$\mathcal F$ is the family of all the singletons of $X$, then 
 $\mathcal F ^ \delta $ turns out to be  the family of all singletons of $X^ \delta $ again,
thus we get back the classical results about ultrafilter convergence in 
products. On the other hand, when $\mathcal F$ is the family of all nonempty subsets of $X$, then $\mathcal F ^ \delta $, in  general, contains certain sets which are not open 
in $X^ \delta $; in fact, $\mathcal F ^ \delta $ is a base for the box topology on 
$X^ \delta $, a topology generally strictly finer than the nowadays standardly used Tychonoff
topology. 

The above fact explains the reason why, in the case
of products, there is not a total symmetry
between results on compactness and results about
pseudocompactness. For example, as already noticed in 
\cite{GS}, it is true that all powers of some topological space $X$ 
are countably compact if and only if $X$ is $D$-compact,
for some ultrafilter $D$ uniform over $ \omega $.
On the other hand, \cite{GS} constructed a topological space $X$ 
all whose powers are pseudocompact, but for which there exists no  ultrafilter 
$D$ uniform  over $ \omega$ such that $X$ is $D$-pseudocompact.
Our framework not only explains the reason for this asymmetry,
but can be used in order to provide a characterization of 
$D$-pseudocompact spaces, a characterization parallel to that
of $D$-compact spaces.
Indeed, we do find  versions 
for $D$-pseudocompactness of the classical results about $D$-convergence (Corollary \ref{pseudofprod}).
Though statements become a little more involved, we believe that
these results have some intrinsic interest.

In Section \ref{sectr} we show that cardinal transfer results for 
decomposable ultrafilters
deeply affect compactness properties relative to these cardinals.
More exactly, if $\lambda$ and $\mu$ are cardinals
such that every uniform ultrafilter
over $\lambda$ is $\mu$-decomposable, then every 
 topological space $X$ which is 
$\mathcal F$-$D$-compact, for some
ultrafilter $D$ uniform over $\lambda$, is also
$\mathcal F$-$D'$-compact, for some
ultrafilter $D'$ uniform over $\mu$.
Of course, this  result applies also to all the equivalent
notions discussed in the preceding sections.
Since there are highly nontrivial set theoretical results 
on transfer of ultrafilter decomposability, our theorems
provide deep unexpected applications of Set Theory 
to compactness properties of products.
The results in Section \ref{sectr}
 generalize some results appeared in \cite{topproc}.

Finally, in Section \ref{mlrelf} we discuss still another
generalization of $[ \lambda , \mu ]$-compactness.
Again, there are relationships with the other compactness properties
introduced before, as well as with 
further variations on pseudocompactness.
The notions introduced in Section \ref{mlrelf} are probably worth of further study.

\section{$D$-compactness relative to some family $\mathcal F$}\label{dcomp}

Suppose that  $D$ is an ultrafilter over some set $Z$, and
$X$ is a topological space.

A family $(x_z)_{z \in Z}$ 
of (not necessarily distinct) elements of $X$ is said to 
$D$-\emph{converge}  to some point $x \in X$ if and only if
$ \{ z \in Z \mid  x_z \in U\} \in D$, for every neighborhood
$U$ of $x$ in $X$.

The space $X$ is said to be $D$-compact if and only if every family
$(x_z)_{z \in Z}$ 
of elements of $X$
converges to some point of $X$.

If
$(Y_z)_{z \in Z}$ is a family 
of (not necessarily distinct) subsets of $X$, then $x$ is called a $D$-\emph{limit point} of $(Y_z)_{z \in Z}$
if and only if 
$ \{ z \in Z \mid  Y_z \cap U \not= \emptyset \} \in D$, for every neighborhood
$U$ of $x$ in $X$.

Since $Y_ z \cap U \not= \emptyset $ if and only if  
$ \overline{Y}_ z \cap U \not= \emptyset $,
we have that 
$x$ is a $D$-limit point of $(Y_z)_{z \in Z}$
if and only if 
$x$ is a $D$-limit point of $( \overline{ Y}_z)_{z \in Z}$.

The space $X$ is said to be $D$-\emph{pseudocompact} 
 if and only if every family
$(O_z)_{z \in Z}$ 
of nonempty open subsets of $X$
has some $D$-limit point in $X$.
The above notion is due to \cite[Definition 4.1]{GS}
for non-principal ultrafilters over $ \omega $,
and appears in \cite{GF} for uniform ultrafilters over arbitrary cardinals. 

The above notions can be simultaneously generalized as follows.

\begin{definition}\label{fcomp}
Suppose that  $D$ is an ultrafilter over some set $Z$, 
$X$ is a topological space, and $\mathcal F$
is a specified family of subsets of $X$.

We say that the space $X$ is  $\mathcal F$-$D$-\emph{compact} if and only if every family
$(F_z)_{z \in Z}$ 
of members of $\mathcal F$ 
has some $D$-limit point in $X$.

Thus, we get the notion of 
$D$-compactness in the particular case when $\mathcal F$ is
the family of all singletons of $X$; and we get
the notion of 
$D$-pseudocompactness in the particular case when $\mathcal F$ is
the family of all nonempty open subsets of $X$.

If $\mathcal G$ is another family of subsets of $X$, let us write
$\mathcal F \rhd \mathcal G$ to mean that,
 for every $F \in \mathcal F$, there is $G \in \mathcal G$
such that $F \supseteq G$.

With this notation, it is trivial to show that
if $\mathcal F \rhd \mathcal G$ and 
$X$ is  $\mathcal G$-$D$-compact,
then $X$ is  $\mathcal F$-$D$-compact.

If $\mathcal F$ is a family of subsets of $X$,
let $ \overline{  \mathcal F} = \{ \overline{F} \mid \ F \in \mathcal F\} $
be the set of all closures of elements of $\mathcal F$.
With this notation, it is trivial to show that
$X$ is  $\mathcal F$-$D$-compact
if and only if $X$ is  $ \overline{\mathcal F}$-$D$-compact.
\end{definition}

The most interesting cases in Definition \ref{fcomp} appear to be the two 
mentioned ones, that is, when either  $\mathcal F$ is the set of all singletons
of $X$, or $\mathcal F$ is the set of all nonempty open subsets of $X$.

In the particular case when  $\mathcal F$ is the set of all singletons,
most of the  results we prove here are essentially  known, except for the 
technical difference that we deal with sequences, rather than subsets 
The difference is substantial only when dealing with
singular cardinals. See Remark \ref{seqversubs} and 
Proposition  \ref{singreg}.
 
In the case when $\mathcal F$ is the set of all nonempty open subsets of $X$, 
most of our results appear to be new. 

\begin{remark} \label{opbase} 
Notice that if $X$ is a topological space,
$\mathcal F$ is the set of all nonempty open subsets of $X$,
and $\mathcal B$ is a base (consisting of nonempty sets)
 for the topology on $X$, then 
both   $\mathcal F \rhd \mathcal B$ and
$\mathcal B \rhd \mathcal F$.

Hence, 
 $\mathcal F$-$D$-compactness
is the same as $\mathcal B$-$D$-compactness.
A similar remark applies to all compactness properties we shall introduce later
(except for those introduced in Section \ref{mlrelf}).
\end{remark}

\section{Complete accumulation points relative to $\mathcal F$} \label{capfs} 

We are now going to generalize  the notion of an accumulation point.

\begin{definition} \label{capf}
If $ \lambda $ is an infinite cardinal, and 
$ (Y _ \alpha ) _{ \alpha \in \lambda } $ is a sequence of subsets of 
some topological space $X$, we say that  $x \in X$ is a
$ \lambda $-\emph{complete accumulation point}  
of
$ (Y _ \alpha ) _{ \alpha \in \lambda } $ 
if and only if 
$ | \{ \alpha \in \lambda \mid Y_ \alpha \cap U \not= \emptyset  \} | = \lambda $,
for every  neighborhood $U$ of $x$ in $X$.  

In case $\lambda= \omega $, we get the usual notion of a \emph{cluster point}.   

Notice that $x $ is a
$ \lambda $-complete accumulation point  
of
$ (Y _ \alpha ) _{ \alpha \in \lambda } $ 
if and only if 
$x $ is a
$ \lambda $-complete accumulation point  
of
$ ( \overline{Y} _ \alpha ) _{ \alpha \in \lambda } $.

If $\mathcal F$
is a family of subsets of $X$, we say that 
$X$ satisfies $\mathcal F$-$\CAP^* _ \lambda $
if and only if 
every sequence
$ (F _ \alpha ) _{ \alpha \in \lambda } $
of members of $\mathcal F$ has a
$ \lambda $-complete accumulation point.  
 \end{definition}  

Notice that if $X$ is a Tychonoff space,
 and $\mathcal F$ is the family of all nonempty open sets
of $X$, then 
a result by Glicksberg 
\cite{Gl},
when reformulated in the present terminology,
asserts that
$\mathcal F$-$\CAP^* _ \omega  $
is equivalent to pseudocompactness.
See also,  e. g., \cite[Section 4]{GS}, \cite{GF, St}.
%%% ?? yyy altri.??? 

If $\mathcal F \rhd \mathcal G$ and 
$X$ satisfies $\mathcal G$-$\CAP^* _ \lambda $, then
$X$ satisfies $\mathcal F$-$\CAP^* _ \lambda $.

Moreover, $X$ satisfies $\mathcal F$-$\CAP^* _ \lambda $ if and only if
it satisfies $ \overline{\mathcal F}$-$\CAP^* _ \lambda $.

\begin{remark} \label{seqversubs}
In the case when each $ Y _ \alpha $ is a singleton
in Definition \ref{capf}, and 
all such singletons are distinct, we get back the usual notion of a
complete accumulation point. 

A point $x \in X$ is said to be a   \emph{complete accumulation point}
of some infinite subset $Y \subseteq X$ if and only if $|Y \cap U|=|Y|$,
for every neighborhood $U$ of $x$ in $X$.

A topological space $X$ satisfies $\CAP_ \lambda $
if and only if every subset $Y \subseteq X$ with $|Y|= \lambda $ 
has a complete accumulation point.   

In the case when $ \lambda $ is a singular cardinal, there is some difference
between the classic notion of a complete accumulation point
and the notion of a $ \lambda $-complete accumulation point,
as  introduced in Definition \ref{capf}.
This happens because, for our purposes, it is more convenient to deal with sequences, rather than
subsets, that is, we allow repetitions. 
This is the reason for the $^*$ in $\mathcal F$-$\CAP^* _ \lambda $
in Definition \ref{capf}.

As pointed in \cite[Part VI, Proposition 1]{nuotop}, if $\mathcal F$ 
is the family of all singletons,  then, for $\lambda$ regular, 
$\mathcal F$-$\CAP^* _ \lambda $ is equivalent to $\CAP _ \lambda $,
and, for $\lambda$ singular, 
$\mathcal F$-$\CAP^* _ \lambda $ is equivalent to the conjunction of 
$\CAP _ \lambda $ and $\CAP _{ \cf\lambda}  $.  

In fact, a more general result holds for families of nonempty sets.
In order to clarify the situation let us introduce the following unstarred variant
of $\mathcal F$-$\CAP ^*_ \lambda $.
If $\mathcal F$
is a family of subsets of $X$, we say that 
$X$ satisfies $\mathcal F$-$\CAP _ \lambda $
if and only if 
every family 
$ (F _ \alpha ) _{ \alpha \in \lambda } $
of \emph{distinct} members of $\mathcal F$ has a
$ \lambda $-complete accumulation point.  
 
Then we have:
\end{remark}   

\begin{proposition} \label{singreg} %3.3
Suppose that 
$X$ is a topological space, and $\mathcal F$
is a  family of   nonempty subsets of $X$.

(a) If  $\lambda$  is a regular cardinal, then $X$ satisfies $\mathcal F$-$\CAP^* _ \lambda $ if and only if 
$X$ satisfies $\mathcal F$-$\CAP_ \lambda $.

(b) If $\lambda$ is a singular cardinal, then $X$ satisfies $\mathcal F$-$\CAP^* _ \lambda $ if and only if 
$X$ satisfies both $\mathcal F$-$\CAP_{ \lambda }$ and $\mathcal F$-$\CAP_{ \cf \lambda }$.
 \end{proposition}  

\begin{proof}
It is obvious that $\mathcal F$-$\CAP^* _ \lambda $ implies
$\mathcal F$-$\CAP_ \lambda $, for every cardinal $\lambda$.

Suppose that $\lambda$ is regular,  that $\mathcal F$-$\CAP_ \lambda $
holds, and that $ (F _ \alpha ) _{ \alpha \in \lambda } $
is a sequence of elements of $\mathcal F$. If some subsequence 
consists
 of $\lambda$-many distinct elements, then, by
 $\mathcal F$-$\CAP_ \lambda $, this subsequence has
 some  $ \lambda $-complete accumulation point
which necessarily is also a $ \lambda $-complete accumulation point
for $ (F _ \alpha ) _{ \alpha \in \lambda } $.
 Otherwise, since $\lambda$ is regular, there exists some $F \in \mathcal F$
which appears $\lambda$-many times in $ (F _ \alpha ) _{ \alpha \in \lambda } $.
Since, by assumption, $F$ is nonempty, just take some $x \in F$ to get a
$ \lambda $-complete accumulation point for
$ (F _ \alpha ) _{ \alpha \in \lambda } $.
Thus we have proved that 
$\mathcal F$-$\CAP_ \lambda $
implies  
$\mathcal F$-$\CAP^* _ \lambda $,
for $\lambda$ regular.

Now
suppose that $\lambda$ is singular and that both $\mathcal F$-$\CAP_ \lambda $
and 
$\mathcal F$-$\CAP _{ \cf \lambda }$
hold.
We are going to show that $\mathcal F$-$\CAP^* _ \lambda $ holds.
Let $ (F _ \alpha ) _{ \alpha \in \lambda } $
be a sequence of elements of $\mathcal F$.
There are three cases.
(i) There exists some $F \in \mathcal F$
which appears $\lambda$-many times in $ (F _ \alpha ) _{ \alpha \in \lambda } $.
In this case, as above, it is enough to choose some element from $F$. 
(ii) Some subsequence of $ (F _ \alpha ) _{ \alpha \in \lambda } $
consists
 of $\lambda$-many distinct elements. Then, as above,
apply $\mathcal F$-$\CAP_ \lambda $ to this subsequence.
(iii) Otherwise, $ (F _ \alpha ) _{ \alpha \in \lambda } $ consists of 
$< \lambda $ different elements, each one appearing $< \lambda $ times.
Moreover, if $ (\lambda _{ \beta }) _{ \beta \in \cf \lambda } $ is a sequence of cardinals $< \lambda $ whose supremum is $\lambda$, then, for every $ \beta \in \cf \lambda $,
there is $F_ \beta \in \mathcal F$ appearing at least  $\lambda _ \beta $-many times.
Since, for each $ \beta $, $F_ \beta $ appears $< \lambda $ times, we can choose 
$ cf \lambda $-many distinct $F_ \beta $'s as above. Applying 
$\mathcal F$-$\CAP _{ \cf \lambda }$ to those  $F_ \beta $'s, 
we get a $ \lambda $-complete accumulation point
for $ (F _ \alpha ) _{ \alpha \in \lambda } $.

It remains to show that
$\mathcal F$-$\CAP^* _ \lambda $ implies
$\mathcal F$-$\CAP _{ \cf \lambda }$.
Let $ (\lambda _{ \beta }) _{ \beta \in \cf \lambda } $ be a sequence of cardinals $< \lambda $ whose supremum is $\lambda$. 
If
$ (F _ \beta  ) _{ \beta  \in \cf\lambda } $
is a sequence of distinct members of $\mathcal F$,
let 
$ (G _ \alpha ) _{ \alpha \in \lambda } $
be a sequence defined in such a way that, for every $ \beta \in \cf \lambda $,
$G_ \alpha =F _ \beta $ for exactly $ \lambda _{ \beta }$-many $\alpha$'s.
By $\mathcal F$-$\CAP^* _ \lambda $, $ (G _ \alpha ) _{ \alpha \in \lambda } $
 has a
$ \lambda $-complete accumulation point $x$. 
It is immediate to show that $x$ is  also a 
$ \cf\lambda $-complete accumulation point for 
$ (F _ \beta  ) _{ \beta  \in \cf\lambda } $.
 \end{proof}

If $D$ is an ultrafilter, $Y$ is a $D$-compact Hausdorff space,
and $X \subseteq Y$, then there is the smallest $D$-compact subspace
$Z$ of $Y$ containing $X$. This is because the intersection
of any family of $D$-compact subspaces of $Y$ is still $D$-compact,
since, in a Hausdorff space, the $D$-limit of a sequence is unique (if it exists).
Such a $Z$ can be also constructed by an iteration procedure in $|I|^+$ stages,
if $D$ is over $I$. This is similar to, e. g., \cite[Theorem 2.12]{GS}, or 
\cite{GF}.

If $X$ is a Tychonoff space, and $Y= \beta (X)$ is the 
Stone-\v Cech compactification of $X$, the smallest $D$-compact subspace
 of $\beta (X)$ containing $X$ is called the 
$D$-\emph{compactification} of $X$, and is denoted by 
$ \beta _D(X)$. 
See, e. g., 
\cite[p. 14]{GF1},
\cite{GF}, or \cite{GS}  for further references and alternative definitions of the  
$D$-compactification (sometimes also called $D$-compact reflection).

\begin{example} \label{exampl}   
(a) If $\lambda$ is singular, then $\cf \lambda$, endowed with 
either the order topology or the discrete topology, fails to satisfy $\CAP _{ \cf \lambda }$, but trivially
satisfies $\CAP _{  \lambda }$.

(b) Suppose that $\lambda$ is singular, and $X$ is any Tychonoff 
space.
If $D$ is an ultrafilter uniform over $\cf \lambda$,
then the $D$-compactification $ \beta _D(X)$ of $X$
satisfies  $\CAP _{ \cf \lambda }$,
by Theorem \ref{equivcpn} (d) $\Rightarrow $ (c) and Proposition 
\ref{singreg} (a). 

(c) If $X$ is $\lambda$  with 
the discrete topology, then $X$ does not satisfy
$\CAP _{  \lambda }$. By (b) above, if 
$D$ is an ultrafilter uniform over $\cf \lambda$, then
the $D$-compactification $ \beta _D(X)$ of $X$
satisfies  $\CAP _{ \cf \lambda }$.
 However, $ \beta _D(X)$ does not satisfy
$\CAP _{  \lambda }$.
Thus, we have a space 
satisfying  $\CAP _{ \cf \lambda }$,
but  not satisfying
$\CAP _{  \lambda }$.

(d) In order to get an example as (c) above, it is not sufficient
to take any space $X$ 
 which does not satisfy  $\CAP _{  \lambda }$.
Indeed, if $X$ is $\lambda$ with 
the order topology,
then  
$ \beta _D(X)$ does satisfy
$\CAP _{  \lambda }$,
if $D$ is an  ultrafilter uniform over $\cf \lambda$.
 \end{example} 

The next proposition shows that, for $\lambda$ a singular cardinal,
$\CAP _{\cf \lambda }$ implies 
$\mathcal F$-$\CAP^* _{  \lambda }$, provided
that $\mathcal F$-$\CAP _{ \mu}$ holds for a set
of cardinals unbounded in $ \lambda $. 

\begin{proposition} \label{singreg2}
Suppose that 
$X$ is a topological space, $\mathcal F$
is a  family of nonempty subsets of $X$, $\lambda$ is a singular cardinal, and
 $(\lambda_ \beta ) _{ \beta \in \cf \lambda }$  is a sequence of cardinals $< \lambda $ such that 
$\sup_{ \beta \in \cf \lambda } \lambda_ \beta = \lambda  $.

If $X$ satisfies $\CAP _{\cf \lambda }$, and  $\mathcal F$-$\CAP _{ \lambda _ \beta }$, for every $ \beta \in \cf\lambda$,  then
$X$ satisfies $\mathcal F$-$\CAP^* _{  \lambda }$.

In particular, 
if $X$ satisfies $\CAP _{\cf \lambda }$, and  
$\CAP _{ \lambda _ \beta }$, for every $ \beta \in \cf\lambda$,  then
$X$ satisfies $\CAP^* _{  \lambda }$.
 \end{proposition}  

\begin{proof}
We first prove that $X$ satisfies $\mathcal F$-$\CAP _{  \lambda }$.
The proof takes some ideas from \cite[proof of the proposition on p. 94]{Sa}.
So, let $ (F _ \alpha ) _{ \alpha \in \lambda } $
be a sequence of distinct elements of $\mathcal F$.
For every $ \beta \in \cf \lambda $,
 by  $\mathcal F$-$\CAP _{ \lambda _ \beta }$,
we get some element $x_ \beta $  
which is a  $ \lambda_ \beta  $-complete accumulation point for
 $ (F _ \alpha ) _{ \alpha \in \lambda_ \beta  } $.
By  $\CAP^* _{\cf \lambda }$
(which follows from $\CAP _{\cf \lambda }$, by Proposition \ref{singreg}(a)),
the sequence $(x_ \beta ) _{ \beta \in \cf \lambda } $
has some $\cf \lambda$-complete accumulation point $x$. 
It is now easy to see that $ x$ is a  
$ \lambda$-complete accumulation point
for $ (F _ \alpha ) _{ \alpha \in \lambda } $.

Since the members of $\mathcal F$ 
are nonempty, $\CAP _{\cf \lambda }$ implies 
$\mathcal F$-$\CAP _{ \cf \lambda }$, hence
$\mathcal F$-$\CAP^* _{  \lambda }$ follows from 
$\mathcal F$-$\CAP _{  \lambda }$, by 
Proposition \ref{singreg}(b).

The last statement follows by taking 
 $\mathcal F$ to be the family of all singletons of $X$. 
 \end{proof} 

The last statement in Proposition 
\ref{singreg2} 
has   appeared in 
\cite[Part VI, p. 2]{nuotop}.

\section{Relationship among compactness properties} \label{rel} 

In the next proposition we deal with the fundamental relationship,
for a given sequence, between
the existence of a $ \lambda $-complete accumulation point  and
the existence  of a $D$-limit point,
for $D$ uniform over $\lambda$. Then in Theorem  \ref{equivcpn} we
shall present more equivalent formulations referring to various compactness
properties.

\begin{proposition} \label{ufacc}
Suppose that $ \lambda $ is an infinite cardinal, and 
$ (Y _ \alpha ) _{ \alpha \in \lambda } $ is a sequence of subsets of 
some topological space $X$.

Then $x \in X$ is a
$ \lambda $-complete accumulation point  
of
$ (Y _ \alpha ) _{ \alpha \in \lambda } $ 
if and only if 
 there exists an ultrafilter $D$ uniform over $ \lambda $
such that 
 $x$ is  a $D$-limit point of $(Y _ \alpha ) _{ \alpha \in \lambda }$.

In particular, 
$ (Y _ \alpha ) _{ \alpha \in \lambda } $ 
has a $ \lambda $-complete accumulation point  
if and only if  
$(Y _ \alpha ) _{ \alpha \in \lambda }$ has a $D$-limit point, for some
ultrafilter $D$ uniform over $ \lambda $.
\end{proposition}

  \begin{proof}
If 
$x \in X$ is a
$ \lambda $-complete accumulation point  
of
$ (Y _ \alpha ) _{ \alpha \in \lambda } $, then the family $\mathcal H$ consisting of the sets
$ \{ \alpha \in \lambda \mid Y_ \alpha \cap U \not= \emptyset  \}$
($U$ a neighborhood of $x$) and $  \lambda \setminus Z $ ($  |Z|< \lambda  $)
has the finite intersection property, indeed,
the intersection of any finite set of members of $\mathcal H$
has cardinality $\lambda$. Hence $\mathcal H$ can be extended to some ultrafilter
$D$, which is necessarily uniform over $ \lambda $. It is trivial to see that, for such a $D$,  $x$ is  a $D$-limit point of $(Y _ \alpha ) _{ \alpha \in \lambda }$.

The converse is trivial, since the ultrafilter $D$ 
is assumed to be uniform over $ \lambda $.
 \end{proof} 

The particular case of Proposition \ref{ufacc} in which all $Y_ \alpha $'s are 
distinct one-element sets is well-known. See \cite[pp. 80--81]{Sa}.

\begin{definition} \label{fcpn}
If $X$ is a topological space, and
$\mathcal F$ is a family of subsets of $X$,
we say that $X$ is
$\mathcal F$-$ [ \mu, \lambda ]$-\emph{compact}
if and only if the following holds.

For every family $( C _ \alpha ) _{ \alpha \in \lambda } $
 of closed sets of $X$, if,
for every $Z \subseteq \lambda $ with $ |Z|< \mu$,
there exists $F \in \mathcal F$ such that   
$ \bigcap _{ \alpha \in Z}  C_ \alpha \supseteq F$,
then  
$ \bigcap _{ \alpha \in \lambda }  C_ \alpha \not= \emptyset $.
\end{definition}   

Of course, in the particular case when $\mathcal F$ is the set of all the singletons,
$\mathcal F$-$ [ \mu,\lambda]$-compactness
is the usual notion of $ [ \mu, \lambda ]$-compactness.

\begin{remark} \label{fol} 
Trivially, if $\mathcal F \rhd \mathcal G$, and 
$X$ is
$\mathcal G$-$ [ \mu, \lambda ]$-compact,
then $X$ is
$\mathcal F$-$ [ \mu, \lambda ]$-compact.

Recall that if $\mathcal F$ is a family of subsets of $X$,
we have defined  $ \overline{  \mathcal F} = \{ \overline{F} \mid \ F \in \mathcal F\} $.
It is trivial to observe that 
$X$ is
$\mathcal F$-$ [ \mu, \lambda ]$-compact if and only if 
$X$ is
$ \overline{ \mathcal F}$-$ [ \mu, \lambda ]$-compact. 
\end{remark}  

\begin{theorem} \label{equivcpn}
Suppose that $X$ is a topological space,
$\mathcal F$ is a family of subsets of $X$,
and $ \lambda $ is a regular cardinal. Then
the following conditions are equivalent.
 
(a) $X$ is $\mathcal F$-$ [ \lambda , \lambda ]$-compact.

(b) Suppose that  $( C _ \alpha ) _{ \alpha \in \lambda } $
is a family of closed sets of $X$ such that $C_ \alpha \supseteq C_ \beta $,
whenever $ \alpha \leq \beta  < \lambda $.
If, for every $ \alpha \in \lambda $, there exists   
$F \in \mathcal F$ such that   
$  C_ \alpha \supseteq F$,
then  
$ \bigcap _{ \alpha \in \lambda }  C_ \alpha \not= \emptyset $.

(b$_1$) Suppose that  $( C _ \alpha ) _{ \alpha \in \lambda } $
is a family of closed sets of $X$ such that $C_ \alpha \supseteq C_ \beta $,
whenever $ \alpha \leq \beta  < \lambda $.
Suppose further that, for every $ \alpha \in \lambda $,
$C _ \alpha $ is the closure of the union of some set of members of $\mathcal F$.   
If, for every $ \alpha \in \lambda $, there exists   
$F \in \mathcal F$ such that   
$  C_ \alpha \supseteq F$,
then  
$ \bigcap _{ \alpha \in \lambda }  C_ \alpha \not= \emptyset $.

(b$_2$) Suppose that  $( C _ \alpha ) _{ \alpha \in \lambda } $
is a family of closed sets of $X$ such that $C_ \alpha \supseteq C_ \beta $,
whenever $ \alpha \leq \beta  < \lambda $.
Suppose further that, for every $ \alpha \in \lambda $,
$C _ \alpha $ is the closure of the union of some set of 
$\leq \lambda $ 
members of $\mathcal F$.   
If, for every $ \alpha \in \lambda $, there exists   
$F \in \mathcal F$ such that   
$  C_ \alpha \supseteq F$,
then  
$ \bigcap _{ \alpha \in \lambda }  C_ \alpha \not= \emptyset $.

(c) Every sequence $ (F _ \alpha ) _{ \alpha \in \lambda } $  of elements
of $\mathcal F$ has a
$ \lambda $-complete accumulation point (that is, $X$ 
satisfies $\mathcal F$-$\CAP^* _ \lambda $).  

(d) For every sequence $ (F _ \alpha ) _{ \alpha \in \lambda } $  of elements
of $\mathcal F$, there exists some ultrafilter $D$ uniform over $\lambda$ such that
 $ (F _ \alpha ) _{ \alpha \in \lambda } $ has a $D$-limit point.

(e) For every $\lambda$-indexed open cover
 $( O _ \alpha ) _{ \alpha \in \lambda } $
 of $X$, there exists $Z \subseteq \lambda $, with $ |Z|< \lambda $,
such that, for every  $F \in \mathcal F$,    
$ F \cap \bigcup _{ \alpha \in Z}  O_ \alpha \not= \emptyset $.

(f) For every $\lambda$-indexed open cover
 $( O _ \alpha ) _{ \alpha \in \lambda } $
 of $X$, such that $O_ \alpha \subseteq O_ \beta $
whenever $ \alpha \leq \beta  < \lambda $, 
there exists $ \alpha \in \lambda$ such that 
$O_ \alpha $ intersects each   $F \in \mathcal F$.

In each of the above conditions we can equivalently replace 
$\mathcal F$ by $ \overline{  \mathcal F}$. 

If $\mathcal F \rhd \mathcal G$ and $\mathcal G \rhd \mathcal F$,
 then in each of the above conditions we can equivalently replace 
$\mathcal F$ by  $\mathcal G$.
 \end{theorem}

\begin{proof}
(a) $\Rightarrow $  (b) is obvious, since $\lambda$ is regular.

Conversely, suppose that (b) holds, and that $( C _ \alpha ) _{ \alpha \in \lambda } $
 are closed sets of $X$ such that,
for every $Z \subseteq \lambda $ with $ |Z|< \mu$,
there exists $F \in \mathcal F$ such that   
$ \bigcap _{ \alpha \in Z}  C_ \alpha \supseteq F$.

For $ \alpha \in \lambda $, define $D_ \alpha =\bigcap _{ \beta < \alpha } C_ \beta  $.
The $ D _ \alpha $'s are closed sets of $X$, and satisfy the assumption in (b),
hence  $ \bigcap _{ \alpha \in \lambda }  D_ \alpha \not= \emptyset $. 
But $ \bigcap _{ \alpha \in \lambda }  C_ \alpha =
 \bigcap _{ \alpha \in \lambda }  D_ \alpha\not= \emptyset $,
thus (a) is proved.

(b) $\Rightarrow $ (b$_1$) $\Rightarrow $ (b$_2$) are trivial.

(b$_2$) $\Rightarrow $ (c) Suppose that (b$_2$) holds, and
that $ (F _ \alpha ) _{ \alpha \in \lambda } $  are  elements
of $\mathcal F$.
For $ \alpha \in \lambda $, 
let $C_ \alpha $ be the closure of 
$\bigcup _{ \beta > \alpha } F_ \beta  $.
The $C_ \alpha $'s satisfy the assumptions in (b$_2$), hence 
$ \bigcap _{ \alpha \in \lambda }  C_ \alpha \not= \emptyset $.
Let $x \in \bigcap _{ \alpha \in \lambda }  C_ \alpha $.
 We want to show that $x$ is a $ \lambda $-complete accumulation point
for $ (F _ \alpha ) _{ \alpha \in \lambda } $.
Indeed, suppose by contradiction that
$ | \{ \alpha \in \lambda \mid F_ \alpha \cap U \not= \emptyset  \} | < \lambda $,
for some  neighborhood $U$ of $x$ in $X$.    
If $ \beta = \sup \{ \alpha \in \lambda \mid F_ \alpha \cap U \not= \emptyset  \}$,
then $ \beta < \lambda $, since $\lambda$ is regular, and we are taking the supremum
of a set of cardinality $<\lambda$.  Thus, $F_ \alpha \cap U = \emptyset $,
for every $ \alpha > \beta $, hence 
$U \cap \bigcup _{ \alpha > \beta } F_ \alpha = \emptyset  $, and 
$x \not\in C_ \beta $, a contradiction.   

(c) $\Rightarrow $ (b)
Suppose that (c) holds, and that $( C _ \alpha ) _{ \alpha \in \lambda } $
satisfies the premise of (b). For each $ \alpha \in \lambda $, choose 
$F_ \alpha \in \mathcal F$ with
$F _ \alpha \subseteq C _ \alpha $. By (c), $ (F _ \alpha ) _{ \alpha \in \lambda } $    has a
$ \lambda $-complete accumulation point $x$. Hence, for every neighborhood 
$U$ of $x$, there are arbitrarily large $ \alpha < \lambda $
such that  $U$ intersects $F_ \alpha $, so  there are arbitrarily large $ \alpha < \lambda $
such that  $U$ intersects $C_ \alpha $, hence 
$U$ intersects every  $C_ \alpha $, since the $C_ \alpha $'s
form a decreasing sequence. In conclusion, for every $ \alpha \in \lambda $, every neighborhood 
of $x$ intersects $C_ \alpha $, that is, $x \in C_ \alpha $,
since $C_ \alpha $ is closed.

(c) $\Leftrightarrow $ (d) is immediate from  
Proposition \ref{ufacc}.

(e) and (f) are obtained from (a) and (b), respectively,
by taking complements. 

It follows from preceding remarks that we get  equivalent
conditions when we replace 
$\mathcal F$ by $ \overline{  \mathcal F}$,
or by $\mathcal G$, if 
 $\mathcal F \rhd \mathcal G$ and $\mathcal G \rhd \mathcal F$. 
\end{proof} 

In the particular case when $\mathcal F$ is the set of all singletons,
 the equivalence of the conditions in Theorem \ref{equivcpn}
(except perhaps for conditions (b$_1$) (b$_2$)) is well-known
and, for the most part, dates back already to 
Alexandroff and Urysohn's classical survey \cite{AU}.
See, e.g., \cite{VLNM, Vfund} for further comments and references.

\begin{remark} \label{omegaps}  
In the particular case when $\lambda= \omega $, 
$X$ is Tychonoff and
 $\mathcal F$ is the family of all nonempty sets of $X$, in
Theorem \ref{equivcpn} 
we get conditions equivalent to pseudocompactness,
since, as we mentioned,  a result by Glicksberg 
implies that, for Tychonoff spaces,
$\mathcal F$-$\CAP^* _ \omega  $
is equivalent to pseudocompactness.
Some of these equivalences are known: for example,
Condition (e) becomes Condition (C$_5$) in
\cite{St}. 
 \end{remark}

\begin{corollary} \label{equivcpncor}
Suppose that $X$ is a topological space,
$\mathcal F$ is a family of subsets of $X$,
and $ \lambda $ is a regular cardinal.
If $X$ is $\mathcal F$-$D$-compact, for some ultrafilter $D$
uniform over $\lambda$, then all the conditions in Theorem \ref{equivcpn}
hold.
 \end{corollary} 

\begin{proof} 
If $X$ is $\mathcal F$-$D$-compact, for some ultrafilter $D$
uniform over $\lambda$, then Condition \ref{equivcpn} (d) holds,
hence all the other equivalent conditions hold.
\end{proof}

\section{Behavior with respect to products} \label{behprod} 

We now discuss the behavior of
$\mathcal F$-$D$-compactness 
with respect to products.

\begin{proposition} \label{prod}
Suppose that $(X_i) _{i \in I} $ 
is a family of topological spaces, and 
let $X= \prod _{i \in I} X_i $, with the Tychonoff topology.
Let $D$ be an ultrafilter  over $\lambda$.

(a) Suppose that,  for each $i \in I$,
$(Y _{i, \alpha }) _{ \alpha \in \lambda } $ is a sequence of subsets of $X_i$.
Then some point $x=(x_ i ) _{ i \in I } $ is a $D$-limit point of 
$( \prod _{i \in I}  Y _{i, \alpha }) _{ \alpha \in \lambda } $ in $X$ 
if and only if, for each $i \in I$,  
$x_i$ is a $D$-limit point of  $(   Y _{i, \alpha }) _{ \alpha \in \lambda } $
in $X_i$.

In particular, 
$( \prod _{i \in I}  Y _{i, \alpha }) _{ \alpha \in \lambda } $ 
has a $D$-limit point 
in $X$ 
if and only if, for each $i \in I$,  
  $(   Y _{i, \alpha }) _{ \alpha \in \lambda } $
has a $D$-limit point in $X_i$.

(b) Suppose that, for each $i \in I$,
$\mathcal F_i$ is a family of subsets of $X_i$,
and let $\mathcal F$ 
be either

- the family of all subsets of $X$ of the form 
$\prod _{i \in I} F_i$, where each $F_i$ belongs to $\mathcal F_i$, or

- for some fixed cardinal $ \nu>1$,
the family of all subsets of $X$ of the form 
$\prod _{i \in I} F_i$, where, for some 
$I' \subseteq I$ with $|I'|<\nu$,  
 $F_i$ belongs to $\mathcal F_i$, for $i \in I'$,
and $F_i = X_i$, for $i \in I \setminus I'$.  

Then $X$ is $\mathcal F$-$D$-compact if and only if 
$X_i$ is 
$\mathcal F_i$-$D$-compact, for every $i \in I$.
 \end{proposition} 

\begin{theorem} \label{fprod} 
Suppose that $X$ is a topological space, and that $\mathcal F$
is a family of subsets of $X$.
For every cardinal $ \delta $, let
$X^ \delta $ be the $ \delta ^{ \text{th} } $ power of $X$,
endowed with the Tychonoff topology, and let 
$\mathcal F^ \delta $ be the family of all products
of $\delta$ members of $\mathcal F$.
Then, for every cardinal $\lambda$, the following are equivalent.
\begin{enumerate} 
\item 
There exists some ultrafilter $D$ uniform over $\lambda$ such that
$X$ is $\mathcal F$-$D$-compact.
\item 
There exists some ultrafilter $D$ uniform over $\lambda$ such that,
for every cardinal $\delta$,
the space $X^ \delta $ is $\mathcal F^ \delta $-$D$-compact.
\item
$X^ \delta $ satisfies 
$\mathcal F^ \delta $-$\CAP^* _ \lambda $,
for every cardinal $\delta$
(if $\lambda$ is regular, then all the equivalent conditions in Theorem  \ref{equivcpn} hold, for $X^ \delta $ and $\mathcal F^ \delta $).
\item
$X^ \delta $ satisfies 
$\mathcal F^ \delta $-$\CAP^* _ \lambda $,
for  $\delta= \min \{ 2 ^{2^ \lambda }, |\mathcal F|^ \lambda \}$ (if $\lambda$ is regular, then all the equivalent conditions in Theorem  \ref{equivcpn} hold,
for $X^ \delta $ and $\mathcal F^ \delta $).
\end{enumerate}   
\end{theorem}  

\begin{proof} 
(1) $\Rightarrow $  (2) follows from Proposition \ref{prod}(b).

(2) $\Rightarrow $  (3) follows from Proposition \ref{ufacc}.

(3) $\Rightarrow $  (4) is trivial.

(4) $\Rightarrow $  (1) We first 
consider the case $ \delta = |\mathcal F|^ \lambda $.
 Thus, there are $ \delta $-many
 $\lambda$-indexed
sequences of elements of $\mathcal F$. Let us enumerate
them as $(F _{\beta , \alpha} ) _{ \alpha \in \lambda } $,
$ \beta $ varying in  $ \delta $.

In $X^ \delta $,
consider the sequence 
$(\prod _{ \beta \in \delta } F _{\beta , \alpha} ) _{ \alpha \in \lambda } $
of elements of $\mathcal F^ \delta $.
By (4), the above sequence has a
$\lambda$-complete accumulation point and, by Proposition 
\ref{ufacc}, there exists some ultrafilter $D$ uniform over $\lambda$ 
such that $(\prod _{ \beta \in \delta } F _{\beta , \alpha} ) _{ \alpha \in \lambda } $
has a $D$-limit point $x$ in  $X^ \delta $.
Say, $x=(x_ \beta ) _{ \beta \in \delta } $.
By Proposition \ref{prod}(a), for every $ \beta \in \delta $,
$ x_ \beta $ is a $D$-limit point of
  $(F _{\beta , \alpha} ) _{ \alpha \in \lambda } $ in $X$.

Since every $\lambda$-indexed
sequence of elements of $\mathcal F$ has the form 
  $(F _{\beta , \alpha} ) _{ \alpha \in \lambda } $,
for some $ \beta \in \delta $, we have that 
 every $\lambda$-indexed
sequence of elements of $\mathcal F$
has some $D$-limit point in $X$, that is, 
$X$ is $\mathcal F$-$D$-compact.

Now we consider the case 
$\delta= 2 ^{2^ \lambda } $.
 We shall prove that
if   $\delta= 2 ^{2^ \lambda } $ and
(1) fails, then  (4) fails.
If (1) fails, then, for every ultrafilter $D$ uniform over $\lambda$,
there is a sequence $(F _ \alpha ) _{ \alpha \in \lambda } $ of elements in 
$\mathcal F$  which has no $D$-limit point.
Since there are $\delta $-many ultrafilters over $\lambda$,
we can enumerate the above sequences 
as $(F _{\beta , \alpha} ) _{ \alpha \in \lambda } $,
$ \beta $ varying in  $ \delta $.

 Now the sequence 
$(\prod _{ \beta \in \delta } F _{\beta , \alpha} ) _{ \alpha \in \lambda } $
in $X^ \delta $
 has no
$\lambda$-complete accumulation point in $X^ \delta $ since, otherwise, by Proposition 
\ref{ufacc}, for some ultrafilter $D$ uniform over $\lambda$, 
it would have some $D$-limit point in  $X^ \delta $. However, this contradicts 
 Proposition \ref{prod}(a) since, by assumption, there is a $ \beta $
 such that $(F _{\beta , \alpha} ) _{ \alpha \in \lambda } $
has no $D$-limit point.
\end{proof} 

\begin{remark} \label{pseudprod} 
Suppose that $\mathcal F$ in Theorem \ref{fprod}
is the family of all nonempty open subsets of $X$.
Then in (3) and (4) we cannot replace $\mathcal F^ \delta $
by the family $\mathcal G^ \delta $ of all nonempty open subsets of $X^ \delta $.
Indeed, if $X$ 
is a Tychonoff space, and
we take $\lambda= \omega $, 
then 
$\mathcal G^ \delta $-$\CAP^* _ \omega  $
for
$X^ \delta $
is equivalent to 
the pseudocompactness of $X^ \delta $.
However, \cite[Example 4.4]{GS} constructed a Tychonoff space 
all whose powers are pseudocompact,
but which for no uniform ultrafilter $D$ over $ \omega $ 
is $D$-pseudocompact. Thus, (3) $\Rightarrow $  (1)
becomes false, in general, if we choose $\mathcal G ^ \delta $
instead of $\mathcal F^ \delta $. 
\end{remark}    

\begin{remark} \label{res}  
In the particular case when $\lambda= \omega $ and $\mathcal F$ is the set of all singletons of $X$, the equivalence of (1), (3) and (4) in Theorem \ref{fprod}
is due to Ginsburg and Saks \cite[Theorem 2.6]{GS}, here
in equivalent form via Theorem \ref{equivcpn}.
See also \cite[Theorem 5.6]{SS} for a related result. 

More generally, when $\mathcal F$ is the set of all singletons of $X$, 
the equivalence of (1) and (3)  in Theorem \ref{fprod} is due to 
\cite[Theorem 6.2]{Sa}.
See also \cite[Corollary 2.15]{GF1}, \cite{Cprepr} and \cite[Theorem 3.4]{C}.
  \end{remark} 

Let us mention the special case of Theorem \ref{fprod}
dealing with $D$-pseudocompactness.

\begin{corollary} \label{pseudofprod} 
Let $X$ be a topological space, and $\lambda$ be an infinite  cardinal.
For every cardinal $ \delta $, let $\mathcal F^ \delta $ be either the family of all
members of $X^ \delta $ which are the products of $ \delta $ nonempty open sets of $X$,
or the family of the nonempty open sets of $X^ \delta $ in the box topology. (Thus, the former family is a base for the topology given by the latter family)  
 Then the following are equivalent.
\begin{enumerate} 
\item 
There exists some ultrafilter $D$ uniform over $\lambda$ such that
$X$ is $D$-pseudocompact.
\item 
There exists some ultrafilter $D$ uniform over $\lambda$ such that,
for every cardinal $\delta$,
every $\lambda$-indexed sequence of members of $\mathcal F^ \delta $ 
has some $D$-limit point in $X^ \delta $ ($X^ \delta $ is endowed with the Tychonoff topology).
\item
For every cardinal $\delta$,
in $X^ \delta $ (endowed with the Tychonoff topology),
every $\lambda$-indexed sequence of 
members of $\mathcal F^ \delta $
has a $\lambda$-complete accumulation point.
\item
Let $\delta= \min \{ 2 ^{2^ \lambda },   \kappa ^ \lambda \}$,
where $ \kappa $ is the weight 
  of $X$. 
In $X^ \delta $ (endowed with the Tychonoff topology),
every $\lambda$-indexed sequence  of
members of $\mathcal F^ \delta $
has a $\lambda$-complete accumulation point.
\item
(provided $\lambda$ is regular)
For every cardinal $\delta$,
$X^ \delta $ (endowed with the Tychonoff topology)
 is $\mathcal F ^ \delta $-$ [ \lambda , \lambda ]$-compact.
\item
(provided $\lambda$ is regular)
Suppose that $ \delta $ is a cardinal,  $( C _ \alpha ) _{ \alpha \in \lambda } $
is a family of closed sets of $X^ \delta $ (endowed with the Tychonoff topology)
and $C_ \alpha \supseteq C_ \beta $,
whenever $ \alpha \leq \beta  < \lambda $.
If, for every $ \alpha \in \lambda $, there exists   
$F \in \mathcal F^ \delta $ such that   
$  C_ \alpha \supseteq F$,
then  
$ \bigcap _{ \alpha \in \lambda }  C_ \alpha \not= \emptyset $.
\end{enumerate}   
\end{corollary}  

\begin{proof}
In order to prove the equivalence 
of conditions (1)-(3), just take $\mathcal F$ in Theorem \ref{fprod}
to be the family of all nonempty sets of $X$, 
 to get the result when 
$\mathcal F^ \delta $ is the family of all
members of $X^ \delta $ which are the products of nonempty open sets of $X$.

In order to get the right bound in Condition (4),
recall that if $\mathcal B$ is a base (consisting of nonempty sets) of $X$,
then, by   
Remark \ref{opbase}, 
  $\mathcal F \rhd \mathcal B$ and
$\mathcal B \rhd \mathcal F$.
Notice also that
$\mathcal F^ \delta  \rhd \mathcal B^ \delta $ and
$\mathcal B^ \delta  \rhd \mathcal F^ \delta $ as well.
Thus, we can apply Theorem \ref{fprod} with
$\mathcal B $ in place of $  \mathcal F $, getting the right bound
in which 
$ |\mathcal B|=\kappa $ is the weight 
  of $X$.

If $\mathcal F'^ \delta $
is the family of the open sets of $X^ \delta $ in the box topology,
then, by Remark \ref{opbase}, trivially both $\mathcal F'^ \delta \rhd \mathcal F^ \delta $
and $\mathcal F'^ \delta \rhd \mathcal F^ \delta $, thus the corollary
holds for $\mathcal F'^ \delta $, too.

If $\lambda$ is regular, then Conditions (5) and (6) are equivalent to
 (3), by Theorem \ref{equivcpn}.
\end{proof}

When $\lambda$ is regular, we
 can use Theorem \ref{equivcpn} in order to get still more conditions equivalent to (3) and (4)
above. 
 
%yyy noto per lambda=omega???? non lo trovo ne in GS ne' nei GF (tranne forse some gen of pseud) 

\section{Two cardinals transfer results} \label{sectr} 

We are now going to show 
that there are very non trivial cardinal transfer properties for
the conditions dealt with in Theorem \ref{fprod}.

Let $D$ be an ultrafilter over $\lambda$, and let 
$f: \lambda \to \mu$. The ultrafilter $f(D)$ over $\mu$
is defined by $Y \in f(D)$ if and only if $f ^{-1}(Y) \in D $.

 \begin{fact} \label{proj} 
Suppose that $X$ is a topological space, $\mathcal F$
is a family of subsets of $X$, 
$D$ is an ultrafilter over $\lambda$,
and $f: \lambda \to \mu$.

If 
$X$ is  $\mathcal F$-$D$-compact,
then  
$X$ is  $\mathcal F$-$f(D)$-compact,
\end{fact} 

If $D$ is an ultrafilter over some set $Z$, and $\mu$
is a cardinal, $D$ is said to be $\mu$-decomposable if and only if
there exists a function $f: Z \to \mu$ such that $f(D)$
is uniform over $\mu$.

 The next corollary implies that if every ultrafilter uniform over $\lambda$ 
is $\mu$-decomposable and the conditions in
 Theorem \ref{fprod} hold for the cardinal $\lambda$, then they hold 
for the cardinal $\mu$, too. 

\begin{corollary} \label{transfer}
Suppose that $\lambda$ is an infinite cardinal, and $K$ 
 is a set of infinite cardinals,
and suppose that every uniform ultrafilter over $\lambda$ is
$\mu$-decomposable, for some $ \mu \in K$.

If $X$ is a topological space, 
$\mathcal F$ is a family of subsets of $X$ and
 one (and hence all) of the conditions in  Theorem \ref{fprod} 
 hold for 
$\lambda$, then there is 
$\mu \in K$ such that the  conditions in  Theorem \ref{fprod}
 hold 
when $\lambda$ is everywhere replaced by $\mu$. 

The same applies with respect to  Corollary  \ref{pseudofprod}.
\end{corollary} 

\begin{proof}
Suppose that the
conditions in  Theorem \ref{fprod} hold for 
$\lambda$.
By  Condition \ref{fprod} (1),
there exists some ultrafilter $D$ uniform over $\lambda$ such that
$X$ is $\mathcal F$-$D$-compact. By assumption,
there exist $\mu \in K$ and $f: \lambda \to \mu$ such that  
$D'=f(D)$ is uniform over $\mu$. By Fact \ref{proj},
  $X$ is $\mathcal F$-$D'$-compact,
hence Condition \ref{fprod} (1) holds for the cardinal $\mu$. 
 \end{proof}  

There are many results asserting that, for some cardinal $\lambda$ 
 and some set $K$, the assumption in Corollary \ref{transfer} holds.
In order to state some of these  results  in a more concise way,
let us denote by 
$\lambda \stackrel{\infty\ }{\Rightarrow} K$, for $K$ 
a set of infinite cardinals, the statement
that the assumption in Corollary \ref{transfer} holds.
That is, 
$\lambda \stackrel{\infty\ }{\Rightarrow} K$
 means that 
every uniform ultrafilter over $\lambda$ is
$\mu$-decomposable, for some $ \mu \in K$.
In the case when
$K = \{ \mu \} $, we simply write  
$\lambda \stackrel{\infty\ }{\Rightarrow} \mu$ in place of
$\lambda \stackrel{\infty\ }{\Rightarrow} K$.
The reason for the superscript $ \infty$ is only to keep the notation consistent with
the notation used in former papers (e. g. \cite{nuotop}).
Notice that many conditions 
equivalent to $\lambda \stackrel{\infty\ }{\Rightarrow} K$
can be obtained from \cite[Part VI, Theorems 8 and 10]{nuotop},
by letting $ \kappa = 2^ \lambda $ there (equivalently, letting
$ \kappa $ be arbitrarily large) there.  

The following are trivial facts about the relation
$\lambda \stackrel{\infty\ }{\Rightarrow} K$.
If $ \lambda \in K$, then $\lambda \stackrel{\infty\ }{\Rightarrow} K$ holds. In particular, 
$\lambda \stackrel{\infty\ }{\Rightarrow}  \lambda $
holds. If $\lambda \stackrel{\infty\ }{\Rightarrow} K$ holds, and $K' \supseteq K$, then $\lambda \stackrel{\infty\ }{\Rightarrow} K'$ holds, too.

In the next Theorem we reformulate, according to the present terminology,
some of the results on decomposability of ultrafilters collected in \cite{mru}. 
In order to state the   theorem, we need to introduce some notational conventions. By $ \lambda ^{+n} $ we denote the $n ^{\rm th}$ successor of $\lambda$, that is,
$ \lambda ^{+n} = \lambda ^{\underbrace{+ \dots +} _{n \ {\rm times}} } $.
%We  put $ \lambda ^{<\mu}  = \sup \{ \lambda ^{ \mu'} \mid \mu'< \mu \} $.
By $\beth_n(\lambda )$ we denote the $n^{\rm th} $ iteration of the power set of
$\lambda$; that is, $\beth_0(\lambda)=\lambda $, and
$\beth_{n+1}(\lambda)=2^{\beth_n(\lambda)}$.
As usual, $[\mu, \lambda ]$ denotes the interval $ \{ \nu \mid \mu \leq \nu \leq \lambda  \} $.
%, and $\Reg$ denotes the class of all regular cardinals.  

\begin{theorem} \label{ufdec}
The following hold.
\begin{enumerate} 

\item If $ \lambda $ is a  regular cardinal, then 
$ \lambda ^+ \stackrel{\infty\ }{\Rightarrow} \lambda $.

\item More generally, if $ \lambda $ is a regular cardinal, then
$ \lambda ^{+n} \stackrel{\infty\ }{\Rightarrow} \lambda $.

\item If $\lambda$ is a singular cardinal, then 
$ \lambda  \stackrel{\infty\ }{\Rightarrow} \cf \lambda $.

\item 
If $\lambda$  is a singular
cardinal, then
$\lambda^+ \stackrel{\infty\ }{\Rightarrow} \{\cf \lambda  \}\cup  K$, for every 
 set $K$  of regular cardinals $ < \lambda $ such that  $K$ is cofinal in $ \lambda $. 

\item %5
$\nu^{\kappa^{+n}} \stackrel{\infty\ }{\Rightarrow} [\kappa,\nu^\kappa ]$.

\item 
If $m\geq 1$, then 
$\beth_m{(\kappa ^{+n}}) \stackrel{\infty\ }{\Rightarrow} [\kappa,2^\kappa ]$.

\item
If $\kappa $ is a strong limit cardinal, then
$\beth_m(\kappa^{+n} )
\stackrel{\infty\ }{\Rightarrow}
\{ \cf\kappa  \} \cup [ \kappa ',\kappa ) $, for every $ \kappa ' <\kappa$.

\item %8
 If $\lambda$ is smaller than the first measurable cardinal 
(or no measurable cardinal exists), then 
$ \lambda  \stackrel{\infty\ }{\Rightarrow} \omega $.

\item More generally, for every infinite cardinal $\lambda$, we have that
$ \lambda  \stackrel{\infty\ }{\Rightarrow} \{  \omega\} \cup M$, where $M$ is the set of all measurable cardinals $\leq \lambda $.

\item %10
 If there is no inner model with a measurable cardinal, and $\lambda \geq \mu$ are  infinite cardinals, then $\lambda \stackrel{\infty\ }{\Rightarrow} \mu$.

\end{enumerate}   

 In particular, Corollary 
\ref{transfer} applies in each of the above cases.
  \end{theorem}

 \begin{remark} \label{bymru}   
 Notice that, by \cite[Properties 1.1(iii),(x)]{mru}, and arguing as in 
\cite[Consequence 1.2]{mru}, the relation
 $\lambda \stackrel{\infty\ }{\Rightarrow} \mu$ is equivalent to
``every $\lambda$-decomposable ultrafilter  
is $\mu$-decomposable''.

Similarly, 
$\lambda \stackrel{\infty\ }{\Rightarrow} K$
is equivalent to
``every $\lambda$-decomposable ultrafilter 
is $\mu$-decomposable, for some $ \mu \in K$''.
 \end{remark}

\begin{proof}[Proof of Theorem \ref{ufdec}]
(1)-(4) and (8)-(9) are immediate from classical results about ultrafilters; see, e. g., 
the comments after Problem 6.8 in \cite{mru}.

(5)-(7) follow from \cite[Theorem 4.3 and Property 1.1(vii)]{mru}. 

(10)  is immediate from  \cite[Theorem 4.5]{Do}, by using 
 \cite[Properties 1.1 and Remark 1.5(b)]{mru}.
\end{proof}

By Remark \ref{bymru}, we get the following transitivity properties of the relation
$\lambda \stackrel{\infty\ }{\Rightarrow} K$.

\begin{proposition} \label{ufdec2}
The following hold.
\begin{enumerate} 

\item
 If $\lambda \stackrel{\infty\ }{\Rightarrow} \mu$
and $\mu \stackrel{\infty\ }{\Rightarrow} K$, then
$\lambda \stackrel{\infty\ }{\Rightarrow} K$.

\item
 More generally, suppose that $\lambda \stackrel{\infty\ }{\Rightarrow} K$
and, for every $\mu \in K$, it happens that
$\mu \stackrel{\infty\ }{\Rightarrow} H_\mu$,
for some set $H_\mu$ depending on $\mu$.
Then    $\lambda \stackrel{\infty\ }{\Rightarrow} \bigcup _{\mu \in K} H_\mu$.

\item 
Suppose that $\lambda \stackrel{\infty\ }{\Rightarrow} K$,
 $\mu \in K$, 
and $\mu \stackrel{\infty\ }{\Rightarrow} K'$,
for some set $K' \subseteq K$ such that $\mu \not\in K'$.
Then    $\lambda \stackrel{\infty\ }{\Rightarrow} K \setminus \{ \mu \} $ .
 
\item
 More generally, suppose that $\lambda \stackrel{\infty\ }{\Rightarrow} K$,
 $H \subseteq  K$ and, for every $\mu \in H$,  
it happens that $\mu \stackrel{\infty\ }{\Rightarrow}K\setminus H$.
Then    $\lambda \stackrel{\infty\ }{\Rightarrow} K \setminus H $ .
\end{enumerate}   
 \end{proposition}

\begin{proof}
(1) and (2) follow from Remark \ref{bymru}.

(4) is immediate from (2), by taking $H_\mu= K \setminus H$, if $\mu \in  H$,
and taking $H_\mu= \{ \mu \} $, if $\mu \in K \setminus H$, since, trivially 
$\mu \stackrel{\infty\ }{\Rightarrow} \mu $.

(3) is a particular case of (4), since $K' \subseteq K \setminus \{ \mu \} $.
\end{proof}

\begin{corollary} \label{corufdec} 
Suppose that $ \kappa  < \nu$ are infinite cardinals,
and that either $K= [\kappa, \nu]$, or  $K= [\kappa, \nu)$. 

(a) If $\lambda \stackrel{\infty\ }{\Rightarrow} K$, then
$\lambda \stackrel{\infty\ }{\Rightarrow} S$, where $S$ is the set
containing $\kappa$, containing all limit cardinals of $K$, and containing all
cardinals of $K$ which are successors of singular cardinals. 

(b) More generally, if $\lambda \stackrel{\infty\ }{\Rightarrow} K$, then
$\lambda \stackrel{\infty\ }{\Rightarrow} L$, where $L$ is the set
of all $  \mu \in K$ such that  either 
\begin{enumerate} 
\item
$\mu= \kappa$, or
\item
$\mu$ is singular and  $\cf \mu < \kappa $, or
\item
$\mu=\varepsilon ^+$, for some singular $\varepsilon $ such that
  $\cf \varepsilon  < \kappa $, or 
\item
$\mu$ is weakly inaccessible.
  \end{enumerate} 

In particular, the above statements can be used to refine
Theorem \ref{ufdec}(5)-(6). 
\end{corollary} 

\begin{proof}
Clearly, (a) follows from (b).
In order to prove (b),
let $H=K\setminus L$,  
thus $L=K \setminus H$.

By Proposition \ref{ufdec2}(4), it is enough to   show that if $\mu \in H$, then 
$\mu \stackrel{\infty\ }{\Rightarrow} L$.

This is trivial if $H= \emptyset $. 
Otherwise, 
suppose by contradiction that there is some 
$\mu \in H$ such that 
$\mu \stackrel{\infty\ }{\Rightarrow} L$ fails.
Let $\mu_0$ be the least such $\mu$.

We now show that there is some $\mu'< \mu_0$ such that
$\mu'\geq \kappa $ and $\mu_0 \stackrel{\infty\ }{\Rightarrow} \mu'$.
This follows from Theorem \ref{ufdec}(1), if $\mu_0$ is the successor of some regular cardinal,
since $\mu_0 > \kappa \not\in H$, by Clause (1).
The existence of $\mu'$ 
follows from Theorem \ref{ufdec}(4), if $\mu_0= \varepsilon ^+$ with $\varepsilon $ singular
such that  $\cf \varepsilon \geq \kappa $. Finally,
the existence of $\mu'$ 
follows from Theorem \ref{ufdec}(3), if $\mu_0$ is singular and $\cf \mu_0 \geq \kappa $.
By Clauses (2)-(4), no other possibility can occur for $\mu_0$, since $\mu_0 \in H $, that is,
$\mu_0 \not \in L$.  

Since $ \kappa \leq \mu' < \mu_0$, then 
$\mu' \stackrel{\infty\ }{\Rightarrow} L$.
This is trivial if $\mu' \in L$; and follows from the minimality
of $\mu_0$, if  $\mu' \not\in L$, which means $\mu'\in H= K\setminus L $.

From
$\mu_0 \stackrel{\infty\ }{\Rightarrow} \mu'$,
and  
$\mu' \stackrel{\infty\ }{\Rightarrow} L$,
we infer
$\mu_0 \stackrel{\infty\ }{\Rightarrow} L$,
by applying Proposition \ref{ufdec2}(1).
We have reached the desired contradiction.
\end{proof} 

Some more  results about the relation
$\lambda \stackrel{\infty\ }{\Rightarrow} K$ follow from
results in \cite{mru}. See \cite{dec}.
See also the comments after \cite[Problem 
6.8]{mru}, in particular, for some open problems 
concerning transfer of decomposability for ultrafilters.

\smallskip

In the particular case when $\mathcal F$ is the set of all singletons, many versions
of Corollary \ref{transfer}  are known, 
 and are usually
stated by means of conditions involving $ [ \lambda , \lambda ]$-compactness
(for regular cardinals, the conditions are equivalent by Theorem  \ref{equivcpn}).
Caicedo \cite{Cprepr} and \cite[Corollary 1.8(ii)]{C}  
proved, among other,  that every productively 
$ [ \lambda ^+, \lambda ^+] $-compact family of
topological spaces is productively $ [ \lambda , \lambda ] $-compact.
More generally, among other, we proved in \cite[Theorem 16]{topappl} that if a product
of topological spaces is
$ [ \lambda ^+, \lambda ^+] $-compact, then all but at most 
$\lambda$ factors are
 $ [ \lambda , \lambda ] $-compact.
Results related to Corollary \ref{transfer}
appear in \cite{Cprepr,C,topproc} and    
\cite[Corollary 4.6]{mru}: generally, they deal with $( \lambda ,\mu)$-regularity
of ultrafilters, which is a notion tightly connected to decomposability,
since, for $\lambda$ a regular cardinal, an ultrafilter is 
$\lambda$-decomposable if and only if it is  $( \lambda , \lambda )$-regular.    
Stronger related results appear in \cite{nuotop}, dealing also
with equivalent notions from Model Theory and Set Theory: in particular, see 
\cite[Part VI, Theorem 8]{nuotop}.
Even in the case when $\mathcal F$ is the set of all singletons, 
some consequences of Theorem \ref{ufdec} and 
Corollaries \ref{corufdec} and \ref{transfer} 
appear to be new, particularly, in the case of singular cardinals.

\smallskip
 
Already the special case $ \mu= \omega $ 
for pseudocompactness of Corollary \ref{transfer}  
appears to have some interest.

\begin{corollary} \label{corpiuspec}
Suppose that $\lambda$ is an infinite cardinal,
and suppose that every uniform ultrafilter over $\lambda$ is
$ \omega $-decomposable
(for example, this happens when either $\cf \lambda = \omega $, or
when $\lambda$ is less than the first measurable cardinal, or if there exists no inner model with a measurable cardinal).

Suppose that $X$ is a topological space
satisfying one of the conditions in Corollary \ref{pseudofprod}.
Then $X$ is $D$-pseudocompact, for some ultrafilter $D$ 
uniform over $ \omega $. In particular, if $X$ is Tychonoff, then
$X$ is pseudocompact, and, furthermore, 
all powers of $X$
are pseudocompact.
 \end{corollary} 

\begin{proof}
Immediate from Remark \ref{omegaps}.
\end{proof}
 
Garcia-Ferreira \cite{GF} contains results related to Corollary \ref{corpiuspec}.
In particular, \cite{GF}  analyzes the relationship between $D$-(pseudo)compactness
and $D'$-(pseudo)compactness for various ultrafilters $D$, $D'$.

\section{$ [ \mu,\lambda ]$-compactness relative to a family $\mathcal F$} \label{mlrelf} % sec7

We can generalize the notion of $ [ \mu,\lambda ]$-compactness
in another direction.

\begin{definition} \label{fcpn2}
If $X$ is a topological space, and
$\mathcal G$ is a family of subsets of $X$,
we say that $X$ is
$ [\mu, \lambda ]$-\emph{compact relative to} $\mathcal G$  
if and only if the following holds.

For every family $( G _ \alpha ) _{ \alpha \in \lambda } $
 of elements of $\mathcal G$, if,
for every $Z \subseteq \lambda $ with $ |Z|< \mu$,
$ \bigcap _{ \alpha \in Z}  G_ \alpha \not= \emptyset $,
then  
$ \bigcap _{ \alpha \in \lambda }  G_ \alpha \not= \emptyset $.
\end{definition}   

The usual notion of 
$ [ \mu, \lambda ]$-compactness
can  be obtained from the above definition
in the particular case when $\mathcal G$ is the family
of all closed sets of $X$.

If $\mathcal G$ is the 
family of all zero sets of some Tychonoff space $X$, 
then $X$ is
$ [\omega , \lambda ]$-compact relative to $\mathcal G$ 
if and only if $X$ is $\lambda$-pseudocompact.
See, e. g., \cite{GF, St}
%yyyy altri???
 for results about  
$\lambda$-pseudocompactness, equivalent formulations, and further references.
Notice that \cite{GF} shows that it is possible, under some set-theoretical assumptions, 
 to construct a space which is not $ \omega _1$-pseudocompact, but which is 
$D$-pseudocompact, for some ultrafilter  $D$ uniform over $ \omega _1$. 

\begin{proposition} \label{prlmcp}
Suppose that $X$ is a topological space, and
$\mathcal G$ is a family of  subsets of $X$. Then the following are equivalent.

(a)
$X$ is
$ [ \mu, \lambda ]$-compact relative to $\mathcal G$.  

(b)
$X$ is
$ [ \kappa  , \kappa ]$-compact relative to $\mathcal G$, for every 
$ \kappa $ with $\mu \leq \kappa \leq \lambda $.    
 \end{proposition}  

\begin{proof}
Similar to the proof of the classical result for $ [ \mu, \lambda ]$-compactness, see,
e. g.,
\cite[Proposition 8]{topappl}. 
 \end{proof} 

There is some connection between the compactness properties
introduced in Definitions \ref{fcpn} and \ref{fcpn2}.
In order to deal with the relationship between the two
properties, it is convenient to introduce a common generalization.

\begin{definition} \label{fgcpn}
If $X$ is a topological space, $\mathcal F$ and
$\mathcal G$  are families of  subsets of $X$,
we say that $X$ is
$\mathcal F$-$ [\mu, \lambda ]$-\emph{compact relative to} $\mathcal G$  
if and only if the following holds.

For every family $( G _ \alpha ) _{ \alpha \in \lambda } $
 of elements of $\mathcal G$, if,
for every $Z \subseteq \lambda $ with $ |Z|< \mu$,
there exists $F \in\mathcal F$ such that 
$ \bigcap _{ \alpha \in Z}  G_ \alpha \supseteq F $,
then  
$ \bigcap _{ \alpha \in \lambda }  G_ \alpha \not= \emptyset $.
\end{definition}   
 
Thus, $\mathcal F$-$ [\mu, \lambda ]$-compactness
is
$\mathcal F$-$ [\mu, \lambda ]$-compactness relative to $\mathcal G$,
when $\mathcal G$ is the family of all closed subsets of $X$.

On the other hand, $ [\mu, \lambda ]$-compactness relative to $\mathcal G$
 is $\mathcal F$-$ [\mu, \lambda ]$-compactness relative to $\mathcal G$,
when $\mathcal F$ is the set of all singletons of $X$.

\begin{proposition} \label{cpcp}
Suppose that $\lambda$ and $\mu$ are infinite cardinals,
 and let $ \kappa = \sup \{ \lambda ^{ \mu'} \mid \mu'< \mu \} $. 
Suppose that $X$ is a topological space, and
$\mathcal F$ is a family of  subsets of $X$. 
Let $\mathcal F^*$
($\mathcal F^* _{\leq \kappa } $, resp.)
be the family of all subsets of $X$
which are the closure of the union of
some family of ($\leq \kappa $, resp.) sets in $\mathcal F$.
Then:
\begin{enumerate}
\item   
The following conditions are equivalent.

(a) $X$ is $\mathcal F$-$ [\mu, \lambda ]$-compact.

(b) $X$ is $\mathcal F$-$ [\mu, \lambda ]$-compact relative to $\mathcal F^*$.

(c) $X$ is $\mathcal F$-$ [\mu, \lambda ]$-compact relative to $\mathcal F^* _{\leq \kappa  } $.
\item
Suppose in addition that all members of $\mathcal F$ are  nonempty. 
If $X$ is  $ [\mu, \lambda ]$-compact relative to $\mathcal F^* _{\leq \kappa } $,
then 
$X$ is $\mathcal F$-$ [\mu, \lambda ]$-compact.
\end{enumerate} 
 \end{proposition}  

\begin{proof} 
In (1), the implications (a)  $\Rightarrow $ (b) $\Rightarrow $  (c) are trivial.

In order to show that (c) $\Rightarrow $  (a)  holds, let 
$( C _ \alpha ) _{ \alpha \in \lambda } $ be a family of 
 closed sets of $X$ such that, 
for every $Z \subseteq \lambda $ with $ |Z|< \mu$,
there exists $F_ Z \in \mathcal F$ such that   
$ \bigcap _{ \alpha \in Z}  C_ \alpha \supseteq F_Z$.

For $ \alpha \in \lambda $, let 
$C' _ \alpha $ be the closure of
$ \bigcup _{ \alpha \in Z}  F_Z$.
Clearly, for every $ \alpha \in \lambda $, we have
 $C_ \alpha \supseteq C' _ \alpha $.
Since there are 
$\kappa$ subsets of $\lambda$ of cardinality
$< \mu$, 
that is, we can choose $Z$ in $\kappa$-many ways, we have that 
each $C'_ \alpha $ is the closure of the union of $\leq \kappa$
  elements from $\mathcal F$. Thus we can apply
(c) in order to get 
$ \bigcap _{ \alpha \in \lambda } C_ \alpha \supseteq 
 \bigcap _{ \alpha \in \lambda } C'_ \alpha
\not= \emptyset$. 
 
(2) is immediate from (1) (c) $\Rightarrow $  (a), since if $\mathcal F$  is a family of \emph{nonempty}
subsets of $X$, then  $ [\mu, \lambda ]$-compactness relative to
some family $\mathcal G$ implies
$\mathcal F$-$ [\mu, \lambda ]$-compactness relative to $\mathcal G$.
\end{proof} 

\begin{remark} \label{cfslm}
The value $ \kappa = \sup \{ \lambda ^{ \mu'} \mid \mu'< \mu \} $ in
Proposition \ref{cpcp} can be improved to $\kappa=$ the cofinality of the 
partial order $S_ \mu( \lambda )$ (see \cite{mru}).
 \end{remark}

\bibliographystyle{plain}

\begin{thebibliography}{KaMa}

\bibitem[AlUr]{AU} P. Alexandroff, P. Urysohn, {\em M\'emorie sur les \'espaces topologiques compacts}, Ver. Akad. Wetensch. Amsterdam \textbf{14} (1929), 1-96. 

\smallskip 

\bibitem[Ca1]{Cprepr} X. Caicedo, {\em On productive $[\kappa,\lambda]$-compactness, or  the Abstract Compactness Theorem revisited},  manuscript (1995). 

\smallskip 


\bibitem[Ca2]{C} X. Caicedo, {\em The Abstract Compactness Theorem Revisited}, in {\em Logic and Foundations of Mathematics (A. Cantini et al. editors),} Kluwer Academic Publishers  (1999), 131--141. 


\smallskip 


%\bibitem[CoNe]{CN}  W. Comfort, S. Negrepontis, {\em The Theory of Ultrafilters}, Berlin (1974). 


\bibitem[Do]
 {Do} H.-D. Donder,  \emph{Regularity of ultrafilters and the core model}, Israel
J. Math. {\bf 63}, 289--322 (1988).

\smallskip 



\bibitem
 [Ga1]{GF1} S. Garcia-Ferreira, \emph{Some remarks on initial $\alpha$-compactness, $<\alpha$-boundedness
   and $p$-compactness},
   Topology Proc. {\bf 15}  (1990), 11--28.

\smallskip 


\bibitem
 [Ga2]{GF} S. Garcia-Ferreira, \emph{On two generalizations of 
pseudocompactness},
Topology Proc. {\bf 24} (Proceedings of the 14$^\text{th}$
   Summer Conference on General Topology and its Applications
  Held at Long Island University, Brookville, NY, August 4--8,
1999) (2001), 149--172. 

\smallskip 



\bibitem
[GiSa]{GS} J. Ginsburg and V. Saks,
\emph{Some applications of ultrafilters in topology},
Pacific J. Math. {\bf  57} (1975), 403--418.

\smallskip 


\bibitem 
[Gl] {Gl} 
I. Glicksberg, \emph{Stone-\v Cech compactifications of products}, Trans. Amer. Math. Soc
{\bf 90} (1959),  369--382 .

\smallskip 


%\bibitem[HNV]{EGT} K. P. Hart, J. Nagata, J. E. Vaughan (editors), {\em Encyclopedia of General Topology}, Amsterdam (2003). 



%\bibitem[KuVa] {KV} K. Kunen and J. E. Vaughan (editors), {\em  Handbook of Set Theoretical Topology}, Amsterdam (1984).

\bibitem[Li1] {topproc} P. Lipparini,  {\em Productive $[\lambda,\mu ]$-compactness and regular ultrafilters}, Topology Proceedings \textbf{21} (1996), 161--171.
 
%\bibitem[L5] {abst} P. Lipparini,  {\em Regular ultrafilters and $[ \lambda , \lambda ]$-compact products of topological spaces} (abstract), Bull. Symbolic Logic \textbf{5} (1999), 121.

\smallskip 


\bibitem[Li2]{topappl} P. Lipparini,  {\em Compact factors in finally compact products of topological spaces}, Topology and its Applications \textbf{153} (2006), 1365--1382.

\smallskip 


\bibitem[Li3]{nuotop} P. Lipparini,  {\em Combinatorial and model-theoretical principles related to 
regularity of ultrafilters and compactness of topological spaces. I}, 
arXiv:0803.3498; {\em II.}:0804.1445; {\em III.}:0804.3737; {\em IV.}:0805.1548  
(2008); {\em V.}:0903.4691; {\em VI.}:0904.3104  (2009).

\smallskip 


\bibitem[Li4]{mru} P. Lipparini,  \emph{More   on regular and decomposable ultrafilters in ZFC}, accepted by Mathematical Logic Quarterly, preprint available on 
arXiv:0810.5587 (2008).

\smallskip

\bibitem[Li5]{dec} P. Lipparini,  \emph{Transfer of ultrafilter decomposability}, 
in preparation.


 \smallskip

\bibitem
[Sa]
{Sa} Saks, Victor, \emph{Ultrafilter invariants in
   topological spaces}, Trans. Amer. Math. Soc. \textbf{241} (1978), 79--97.

%\bibitem[Ste]{steph} R. M. Stephenson, {\em Initially $\kappa $-compact
% and related spaces}, Chapter 13 in \cite{KV}, 603--632. 

\smallskip 


\bibitem[ScSt]{SS}  C. T.  Scarborough, A. H: Stone, \emph{Products of     
   nearly compact spaces}, Trans. Amer. Math. Soc. {\bf 124}  (1966),  131--147.  

\smallskip 


\bibitem[St]{St} R. M. Stephenson Jr, \emph{Pseudocompact spaces}, ch. d-07 in 
\emph{Encyclopedia of general topology}, Edited by
   K. P. Hart, J. Nagata and J. E. Vaughan. Elsevier
   Science Publishers, B.V., Amsterdam, 2004.

\smallskip 


\bibitem[Va1]{VLNM} J. E. Vaughan, {\em Some recent results in the theory of  [a,b]-compactness},
in {\em TOPO 72---General Topology and its Applications} 
(Proc. Second Pittsburg Internat. Conf., Carnegie-Mellon Univ. and Univ. Pittsburg, 1972), Lecture Notes in Mathematics 
 \textbf{378} (1974), 534--550. 

\smallskip 


 \bibitem[Va2]{Vfund} J. E. Vaughan, {\em Some properties
 related to [a,b]-compactness}, Fund. Math. \textbf{87} (1975), 251--260. 

\end{thebibliography}

\end{document}